\documentclass[12pt,english]{article}
\usepackage[T1]{fontenc}
\usepackage[latin9]{inputenc}
\usepackage[letterpaper]{geometry}
\usepackage{textcomp}
\usepackage{amsthm}
\usepackage{amsmath}
\usepackage{graphicx}
\usepackage{setspace}
\usepackage{amssymb}
\usepackage[authoryear]{natbib}
\makeatletter

\providecommand{\tabularnewline}{\\}

\newcommand{\lyxaddress}[1]{
\par {\raggedright #1
\vspace{1.4em}
\noindent\par}
}
\theoremstyle{plain}
\newtheorem{thm}{Theorem}
 \theoremstyle{definition}
  \newtheorem{example}[thm]{Example}
  \theoremstyle{plain}
  \newtheorem{lem}[thm]{Lemma}
  \theoremstyle{plain}
  \newtheorem{cor}[thm]{Corollary}

\DeclareMathOperator{\T}{T}

\DeclareMathOperator{\N}{N}

\DeclareMathOperator{\reg}{reg}

\makeatother

\usepackage{babel}

\begin{document}

\title{Statistical inference optimized with respect to the observed sample
for single or multiple comparisons}

\maketitle

\lyxaddress{
\lyxaddress{David R. Bickel\\
Ottawa Institute of Systems Biology\\
Department of Biochemistry, Microbiology, and Immunology\\
Department of Mathematics and Statistics\\
University of Ottawa\\
451 Smyth Road\\
Ottawa, Ontario, K1H 8M5}}
\begin{abstract}
The normalized maximum likelihood (NML) is a recent penalized likelihood
that has properties that justify defining the amount of discrimination
information (DI) in the data supporting an alternative hypothesis
over a null hypothesis as the logarithm of an NML ratio, namely, the
alternative hypothesis NML divided by the null hypothesis NML. The
resulting DI, like the Bayes factor but unlike the p-value, measures
the strength of evidence for an alternative hypothesis over a null
hypothesis such that the probability of misleading evidence vanishes
asymptotically under weak regularity conditions and such that evidence
can support a simple null hypothesis. Unlike the Bayes factor, the
DI does not require a prior distribution and is minimax optimal in
a sense that does not involve averaging over outcomes that did not
occur. Replacing a (possibly pseudo-) likelihood function with its
weighted counterpart extends the scope of the DI to models for which
the unweighted NML is undefined. The likelihood weights leverage side
information, either in data associated with comparisons other than
the comparison at hand or in the parameter value of a simple null
hypothesis. Two case studies, one involving multiple populations and
the other involving multiple biological features, indicate that the
DI is robust to the type of side information used when that information
is assigned the weight of a single observation. Such robustness suggests
that very little adjustment for multiple comparisons is warranted
if the sample size is at least moderate. 
\end{abstract}

\newpage{}

\section{\label{sec:Introduction}Introduction}

\subsection{\label{sub:Quantifying-statistical-evidence}Quantifying statistical
evidence}

Many areas of science involve investigations of whether some effect
is present and thus call for statistical methods that assess the evidence
pertaining to whether a null hypothesis or an alternative hypothesis
is closer to the system studied. For example, many experimental biologists
are more interested in whether gene expression levels differ between
control and treatment groups than in the effect size itself. 

Because not all samples are representative of their populations, the
amount of evidence against the null hypothesis is misleadingly high
for some samples. Although the probability of observing such an unrepresentative
sample should decrease as the size of the sample increases, that is
not the case if proximity of a p-value to 0 is interpreted as the
strength of evidence against the null hypothesis.  Indeed, the distribution
of the p-value associated with a simple (point) null hypothesis remains
the same at all sample sizes if the null hypothesis holds, making
the p-value impossible to interpret as a level of evidence apart from
considering the sample size, as \citet{RefWorks:122}, \citet{Blume2003439},
and others have argued; cf. \citet{RefWorks:293} on the sample-size
incoherence of significance testing. \citet{RefWorks:435} defined the lacking property
by calling a measure of evidence \emph{interpretable }if its probability
of misleading evidence vanishes asymptotically. That is, a measure
of evidence satisfies the interpretability condition only if the frequentist
probability of observing a sample that has misleading evidence exceeding
some fixed threshold converges to 0 as the sample size diverges.

Another adverse consequence of treating the p-value as a measure of
evidence is its inability to indicate evidence in favor of a simple
null hypothesis. In general, the amount of information in the data
that favors a simple null hypothesis cannot be quantified by the p-value
since it can only indicate whether there is evidence against it. 

The Bayes factor in principle overcomes the above limitations of
the p-value but poses the notorious problem of specifying the prior
distribution of a nuisance parameter that is not random in the frequentist
sense. Any solution to the problem has practical implications since
the Bayes factor is sensitive to prior specification \citep{RefWorks:183}. 

Subjective prior distributions have the advantage of coherence and
yet are rarely used in data analysis since they depend on arbitrary
choices in prior specification. On the other hand, the improper prior
distributions generated by conventional algorithms cannot be directly
applied to model selection since they would leave the Bayes factor
undefined. That has been overcome to some extent by dividing the data
into training and test samples, with the training samples generating
proper priors for use with test samples, but at the expense of requiring
the specification of training samples and, in the presence of multiple
training samples, a method of averaging \citep{RefWorks:1023}. Further,
the interpretation of the resulting posterior probability is not
clear except perhaps as an approximation to an agent's level of belief
\citep{RefWorks:17}.

\subsection{\label{sub:Optimality-mean}Repeated-sampling optimality}

Since there are many potential measures of evidence, most notably
the Bayes factors defined by different priors, that satisfy the criteria
that a measure of evidence be interpretable and that it can support
a simple null hypothesis, an optimality criterion will be applied
to determine a unique method of hypothesis testing and more general
model selection. Before doing so, that criterion will be distinguished
from standards of optimality in the received framework of statistics,
that of Neyman and Pearson as generalized by Wald.

The goal of minimizing \emph{risk}, the expected loss with respect
to a sampling distribution \citep{Wald1950}, has provided a unified
framework of estimation and testing and, as briefly reviewed in \citet[§3.1]{conditional2009},
has led to recent multiple comparison procedures. However, \citet{RefWorks:1302},
\citet{FraserAncillaries2004a}, \citet{RefWorks:1334}, and other
frequentist statisticians have criticized the framework for promoting
opportunistic trade-offs between hypothetical samples, thereby potentially
misleading scientists and yielding unacceptably pathological procedures.
The main non-Bayesian alternative involves replacing the marginal
sampling distribution with a conditional sampling distribution given
an exact or approximate ancillary statistic \citep[e.g.,][§3.3]{RefWorks:SprottBook2000}.

While conditioning on an ancillary statistic makes the reference distribution
more relevant to inference on the basis of the observed sample \citep{RefWorks:985},
it still does not permit statements about the actual loss incurred.
For example, the confidence level remains the proportion of confidence
intervals corresponding to repeated samples that cover the parameter
of interest. Although the use of exact confidence intervals minimizes
a risk (\citealp{RefWorks:1187}; \citealp{frequentistReasoning}),
it is silent regarding the loss associated with the observed sample.

\subsection{\label{sub:Optimality-observed}Observed-sample optimality}

\subsubsection{Information-theoretic inference}

In order to address the issues outlined above, this paper continues
the development of a new information-theoretic alternative to previous
approaches to statistical inference. The concept of a predictive distribution
will enable defining minimax optimality without repeated-sampling
or posterior-distribution averages. This approach is presented here
largely without the terminology of its origin in universal source
coding \citep{RefWorks:404}.

Consider the observed data vector $x\in\mathcal{X}^{n}$. Let $\mathcal{E}\left(\Omega\right)$
denote the set of all probability density functions on any sample
space $\Omega$, and let $\mathcal{F}=\left\{ f_{\phi}:\phi\in\Phi\right\} \subset\mathcal{E}\left(\mathcal{X}^{n}\right)$
denote a parametric family of density functions on $\mathcal{X}^{n}$
for parameter space $\Phi$. (Herein, the probability densities are
Radon-Nikodym derivatives, reducing to probability masses if $\mathcal{X}$
is countable.) The maximum likelihood estimate of $\phi$, denoted
by $\hat{\phi}\left(x\right)$, is assumed to be unique. 

The \emph{regret} of a predictive density $\bar{f}\in\mathcal{E}\left(\mathcal{X}^{n}\right)$
is the logarithmic loss\begin{eqnarray}
\reg\left(\bar{f},x;\Phi\right) & = & -\log\bar{f}\left(x\right)-\inf_{\phi\in\Phi}\left(-\log f_{\phi}\left(x\right)\right)=\log\frac{f_{\hat{\phi}\left(x\right)}\left(x\right)}{\bar{f}\left(x\right)}\label{eq:untargeted-regret}\end{eqnarray}
for any $x\in\mathcal{X}^{n}$. The \emph{$\mathcal{E}\left(\mathcal{X}^{n}\right)$-optimal
predictive density function} \emph{relative to }$\mathcal{F}$, \begin{equation}
\bar{f}_{0}=\arg\inf_{\bar{f}\in\mathcal{E}\left(\mathcal{X}^{n}\right)}\sup_{u\in\mathcal{X}^{n}}\reg\left(\bar{f},u;\Phi\right),\label{eq:optimality}\end{equation}
while by definition in $\mathcal{E}\left(\mathcal{X}^{n}\right)$,
is not necessarily in $\mathcal{F}$. Rather, $\bar{f}_{0}$ is a
probability density function that represents the entire family $\mathcal{F}$
with a single distribution, much as does a prior predictive density
function. Instead of averaging the members of $\mathcal{F}$ with
respect to a prior distribution, the present definition employs $\mathcal{F}$
in equation \eqref{eq:optimality} for each $u\in\mathcal{X}^{n}$
through the maximization of the likelihood over $\phi\in\Phi$, as
seen by substituting $u$ for $x$ in equation \eqref{eq:untargeted-regret}. 

Originally motivated in the information theory literature by a need
to minimize codelength \citep{RefWorks:404}, equation \eqref{eq:optimality}
defines the type of minimax optimality employed as opposed to the
optimality of Section \ref{sub:Optimality-mean}. (According to the
minimum description length principle, each family of distributions
corresponds to an algorithm of most efficiently encoding the information
in $x$ \citep{RefWorks:374,RefWorks:375,Rissanen2009b}.) The predictive
density function $\bar{f}_{0}$ is optimal in that it solves the minimax
problem involving all $u\in\mathcal{X}^{n}$, and thus for the observed
sample $x\in\mathcal{X}^{n}$, rather than the more usual minimax
problem involving an expectation value over all samples, as in the
standard decision theory of frequentism. The following result \citep{RefWorks:404,RefWorks:374,RefWorks:375},
to be proved in Section \ref{sub:NMWL} for a more general optimization
problem, sheds light on the nature of the optimality considered.
\begin{thm}
\label{thm:NML}If $\int_{\mathcal{X}^{n}}f_{\hat{\phi}\left(u\right)}\left(u\right)du<\infty$,
then the $\mathcal{E}\left(\mathcal{X}^{n}\right)$-optimal predictive
density function relative to $\mathcal{F}$ is\begin{equation}
\bar{f}_{0}\left(\bullet\right)=\bar{f}_{0}\left(\bullet;\Phi\right)=\frac{f_{\hat{\phi}\left(\bullet\right)}\left(\bullet\right)}{\int_{\mathcal{X}^{n}}f_{\hat{\phi}\left(u\right)}\left(u\right)du}.\label{eq:NML}\end{equation}
\end{thm}
\begin{proof}
This proof by contradiction is based on the direct proof given by
\citet[§6.2.1]{RefWorks:375}. Assume, contrary to the claim, that
the density function $\bar{f}_{0}$ that satisfies equation \eqref{eq:NML}
is not the optimal predictive density function. Since, for any $v\in\mathcal{X}^{n}$,
the ratio $\bar{f}_{0}\left(v\right)/f_{\hat{\phi}\left(v\right)}\left(v\right)$
does not depend on $v$, it follows that, for any $\breve{f}\in\mathcal{E}\left(\mathcal{X}^{n}\right)\backslash\left\{ \bar{f}_{0}\right\} $,
there is a $v\in\mathcal{X}^{n}$ such that $\breve{f}\left(v\right)/f_{\hat{\phi}\left(v\right)}\left(v\right)<\bar{f}_{0}\left(v\right)/f_{\hat{\phi}\left(v\right)}\left(v\right)$.
Therefore, given any $\breve{f}\in\mathcal{E}\left(\mathcal{X}^{n}\right)\backslash\left\{ \bar{f}_{0}\right\} $,
there is a $v\in\mathcal{X}^{n}$ such that $\reg\left(\bar{f}_{0},v;\Phi\right)<\reg\left(\breve{f},v;\Phi\right)$,
which contradicts the assumption. 
\end{proof}
Note that $u$, the dummy variable of integration over $\mathcal{X}^{n}$,
appears twice in the integrand. For the observed $x\in\mathcal{X}^{n}$,
the quantity $\bar{f}_{0}\left(x\right)=\bar{f}_{0}\left(x;\Phi\right)$
is called the \emph{normalized maximum likelihood} (NML) with respect
to $\Phi$. 

According to Theorem \ref{thm:NML}, the minimax optimality \eqref{eq:optimality}
of $\bar{f}_{0}$ guarantees that $\reg\left(\bar{f}_{0},x;\Phi\right)$,
the regret due to the observed sample, cannot exceed $\sup_{u\in\mathcal{X}^{n}}\reg\left(\bar{f}_{0},u;\Phi\right)$,
the regret due to the worst-case sample. In that sense, $\bar{f}_{0}$
is optimal for the observed sample. By contrast, standard frequentist
optimality, concerned only with loss averaged over all possible samples,
guarantees no bound on the loss inflicted by any individual sample. 

Such observed-sample optimality justifies selecting the model or hypothesis
corresponding to the family $\mathcal{F}_{i}$ of distributions that
minimizes $-\log\bar{f}_{0}\left(x;\Phi_{i}\right)$, the observed
prediction error of the $i$th among a finite number of distribution
families under consideration. Following the terminology of \citet{RefWorks:292} and \citet{mediumScale},
$-\log\bar{f}_{0}\left(x;\left\{ \phi_{0}\right\} \right)-\left(-\log\bar{f}_{0}\left(x;\Phi\right)\right)$
would be the \emph{information in $x$ for discrimination} in favor
of the alternative hypothesis that $\phi\ne\phi_{0}$ over the null
hypothesis that $\phi=\phi_{0}$. Such information is an interpretable
measure of evidence under general conditions and can quantify the
strength of any evidence in favor of the null hypothesis as well as
that of any evidence against it. More importantly, the information
for discrimination  optimally quantifies the difference in how well
each model or hypothesis predicts relative to ideal predictors of
individual samples rather than relative to unknown true distributions,
the ideal predictors in the sense of averages over samples. The Kullback-Leibler
risk, for example, only measures mean discrimination information relative
to unknown ideal predictors in the average sense. 

Since the base of the logarithm is inconsequential, it may be chosen
for convenience of interpretation. The binary logarithm $\left(\log_{2}\right)$,
yielding the number of \emph{bits} of information, enables not only
immediate exponentiation back to the ratio domain but also the use
of grades of evidence that are both broad enough and refined enough
for applications across scientific disciplines (Table \ref{tab:Grades-of-hypothesis-evidence}).
Except for the distinction between negligible and weak evidence, the
grades closely mirror those \citet{RefWorks:182} originally proposed
for the Bayes factor; cf. \citet{RefWorks:435}. Accordingly, the
$\left[3,5\right)$ grade of Table \ref{tab:Grades-of-hypothesis-evidence}
is what \citet[§1.12]{RefWorks:122} considers {}``fairly strong
evidence'' for one simple hypothesis over another, and the $\left[5,7\right)$
and $\left[7,\infty\right)$ grades together constitute his {}``quite
strong evidence.'' %
\begin{table}
\begin{tabular}{|c|c|c|c|c|c|c|}
\hline 
\textbf{Information (bits)} & $\left(0,1\right)$ & $\left[1,2\right)$ & $\left[2,3\right)$ & $\left[3,5\right)$ & $\left[5,7\right)$ & $\left[7,\infty\right)$\tabularnewline
\hline 
\textbf{Evidence grade} & Negligible & Weak & Moderate & Strong & Very strong & Overwhelming\tabularnewline
\hline
\end{tabular}

\caption{\label{tab:Grades-of-hypothesis-evidence}Heuristic grades of evidence
for an alternative hypothesis over a null hypothesis corresponding
to intervals of the information for discrimination. The absolute value
of a negative amount of information gives the grade of evidence favoring
the null hypothesis.}

\end{table}

\subsubsection{Extension of information-theoretic inference}

Despite the unique observed-sample optimality of the NML for quantifying
discrimination information, three shortcomings make it impractical
for use in many biostatistics applications. First, since such applications
typically partition $\phi$ into an interest parameter $\theta$ and
a nuisance parameter $\lambda$, the regret is relative to an ideal
distribution determined by maximizing the likelihood not only over
$\theta$ but also over $\lambda$. As a result, the ideal member
of the family of distributions would be considered a better predictor
than another member that has the same value of $\theta$ on the basis
of having a different value of $\lambda$, which should be irrelevant.
Thus, the NML is inadequate for testing hypotheses about $\theta$
in the presence of $\lambda$.

Second, the NML only uses information that is in $x$, but considering
such information about the parameter in isolation from other available
information can be misleading unless the sample size is sufficiently
large. Additional information may be available in data from other
populations, from other biological features such as genes or SNPs,
or from other feature-feature comparisons. Even in the absence of
such incidental information, there would be some information in the
fact that the null hypothesis that $\phi=\phi_{0}$ is seriously considered. 

Third, the normalizing denominator of equation \eqref{eq:NML}, the
logarithm of which is called the \emph{parametric complexity} of $\mathcal{F}$,
is infinite for typical families of distributions, including the normal
family. Each of the variant NMLs proposed to address the problem introduces
its own conceptual difficulties \citep{Lanterman2005b,RefWorks:375}.
For example, \citet[§5.2.4]{RefWorks:374}, \citet{RissanenRoos2007b},
and \citet[§11.4.2]{RefWorks:375} proposed conditional versions of
the NML. Cf. related work by \citet{TakimotoWarmuth2000b}.

To overcome the first of the three identified problems with NML, it
is generalized in Section \ref{sec:NMWL} by replacing the original
data with a statistic that is a function of the data and that has
a distribution depending on $\theta$ but not on $\lambda$. Since
the information in the data relevant to the interest parameter is
largely confined to the statistic, that information can be better
quantified in terms of the distribution of the statistic than in terms
of the distribution of the original data, the latter depending on
the value of the nuisance parameter. In terms of the minimum description
length (MDL) metaphor \citep{RefWorks:374,RefWorks:375}, the data
are first compressed with little information loss by reduction to
a smaller-dimensional statistic and then further compressed by the
family of distributions. 

The use of a weighted likelihood addresses the second problem in the
same section, which also includes some results relevant to the probability
of observing misleading information. (The weighted likelihood was
originally proposed for bias-variance trade-offs given relatively
small $n_{i}$ but potentially large $N$ \citep{Feifang2002347}.
More formally, \citet{RefWorks:509} derived the weighted likelihood
from the minimization of Kullback-Leibler loss.)

As a by-product for commonly used distribution families, that solution
to the second problem automatically solves the third problem, as illustrated
in Section \ref{sec:Case-study} with a multiple-population data set
and a multiple-feature data set. Finally, Section \ref{sec:Discussion}
concludes by highlighting desirable properties of the new NML-based
measure of information for discrimination.

\section{\label{sec:NMWL}Optimal inference}

\subsection{\label{sub:Preliminaries}Preliminaries}

\subsubsection{Weighted likelihood}

The framework of Section \ref{sub:Optimality-observed} is generalized
by the use of data reduction to eliminate a nuisance parameter in
$\phi$. Consider a measurable map $\tau_{n}:\mathcal{X}^{n}\rightarrow\mathcal{T}\left(n\right)$.
Let $\theta:\Phi\rightarrow\Theta$ denote a subparameter function
such that the probability density of $\tau_{n}\left(X\right)$ is
$g_{\theta\left(\phi\right)}\left(\tau_{n}\left(X\right)\right)$,
abbreviated as $g_{\theta}\left(\tau\left(X\right)\right)$; the dependence
of the density function $g_{\theta}$ on $n$ is suppressed. Thus,
the reduction of the data $X$ to a statistic $\tau\left(X\right)$
has the effect of replacing the full parameter $\phi$ with the interest
parameter $\theta$. Important special cases of $L\left(\theta;\tau\left(x\right)\right)=g_{\theta}\left(\tau\left(x\right)\right)$
as a function of $\theta$ are conditional likelihood functions and
marginal likelihood functions (\citealp{RefWorks:122,Severini2000}; \citealp{mediumScale}). 

The framework is now extended to $N$ hypotheses or comparisons. Let
$\mathcal{G}_{n,i}=\left\{ g_{n,i,\theta}:\theta\in\Theta\right\} \subset\mathcal{E}\left(\mathcal{T}\left(n\right)\right)$
denote the parametric family of density functions on $\mathcal{T}\left(n\right)$
for parameter space $\Theta$. The {}``$n$'' and {}``$i$'' subscripts
will be dropped when their values are clear. For the $i$th of $N$
null hypotheses or comparisons, suppose $x_{i}\in\mathcal{X}^{n_{i}}$
is a realization of the random vector $X_{i}$ of $n_{i}$ independent
components. Then each $T_{i}=\tau\left(X_{i}\right)$ is distributed
with density $g_{\theta_{i}}=g_{n_{i},i,\theta_{i}}$, and each outcome
$t_{i}=\tau\left(x_{i}\right)$ is an element of $\mathcal{T}_{i}=\mathcal{T}\left(n_{i}\right)$.
Let $L_{i}\left(\theta;\tau\left(x_{i}\right)\right)=g_{n,i,\theta}\left(\tau\left(x_{i}\right)\right)=g_{\theta}\left(\tau\left(x_{i}\right)\right)$,
giving each comparison its own likelihood function. 

Mapping $\mathcal{X}^{n_{i}}$ to $\mathcal{T}_{i}=\mathbb{R}^{D}$
is common in data reduction applications in which $\Theta=\mathbb{R}^{D}$.
Assigning a common parametric family to all comparisons ($\mathcal{G}_{n,i}=\mathcal{G}_{n,1}$
for all $i$) is usually appropriate when each comparison corresponds
to a biological feature, as in Section \ref{sub:Biological-features}.

The observation $\mathbf{x}=\left\langle x_{1},\dots,x_{N}\right\rangle $
generates the test statistic vector $\mathbf{t}=\left\langle t_{1},\dots,t_{N}\right\rangle =\left\langle \tau\left(x_{1}\right),\dots,\tau\left(x_{N}\right)\right\rangle $,
an outcome of $\mathbf{T}=\left\langle T_{1},\dots,T_{N}\right\rangle =\left\langle \tau\left(X_{1}\right),\dots,\tau\left(X_{N}\right)\right\rangle $.
For inference about $\theta_{i}$ on the basis of $\mathbf{t}$, the
\emph{weighted likelihood function} $\bar{L}_{i}\left(\bullet;\mathbf{t}\right):\Theta\rightarrow\left[0,\infty\right)$
is defined by \begin{equation}
\log\bar{L}_{i}\left(\theta_{i};\mathbf{t}\right)=\sum_{j=1}^{N}w_{ij}\log L_{j}\left(\theta_{i};t_{j}\right),\label{eq:weighted-likelihood}\end{equation}
where the weights $w_{i}=\left\langle w_{i1},\dots,w_{iN}\right\rangle $
are real numbers that may depend on $\left\langle n_{1},\dots,n_{N}\right\rangle $
and that satisfy $w_{ii}\ge w_{ij}$ \citep{Feifang2002347}. The
weights normally also conform to $\sum_{j=1}^{N}w_{ij}=1$, a requirement
that will be temporarily relaxed in Section \ref{sub:NMWL}.
\begin{example}
\label{exa:microarray}In most microarray studies, the expression
levels of $N$ genes are measured with the goal of determining which
genes are differentially expressed between a treatment/perturbation
group of $m$ replicates and a control group of $n$ replicates; each
of these biological replicates represents one or more organisms. (Single-channel
arrays do not require the pairing of replicates between groups as
did the dual-channel arrays.) Following the typical assumption that
intensity values are lognormally distributed, let $x_{i}=\left\langle x_{i1},\dots,x_{im}\right\rangle $
and $y_{i}=\left\langle y_{i1},\dots,y_{in}\right\rangle $ denote
the logarithms of the $m$ and $n$ intensities of the $i$th gene
in the perturbation and control group, respectively. For small numbers
of replicates, the assumption of a common variance within each group
is useful: $X_{ij}\sim\N\left(\xi_{i},\sigma_{i}^{2}\right)$ and
$Y_{ij}\sim\N\left(\eta_{i},\sigma_{i}^{2}\right)$ with realized
values $X_{ij}=x_{ij}$ for $j=1,\dots,m$ and $Y_{ij}=y_{ij}$ for
$j=1,\dots,n$. If $\theta_{i}$ is the absolute value of the \emph{inverse
coefficient of variation} $\left(\xi_{i}-\eta_{i}\right)/\sigma_{i}$,
then $t_{i}$ is conveniently taken as the absolute value of the two-sample,
equal-variance $t$-statistic, which has a noncentral $t$ distribution
with noncentrality parameter $\left(m^{-1}+n^{-1}\right)^{-1/2}\theta_{i}$
and $m+n-2$ degrees of freedom.
\end{example}
The sampling distribution of $\mathbf{T}$ is denoted by $P$ to specify
properties of the weights while accommodating model misspecification,
the case that there is not a $\theta_{i}\in\Theta$ such that $g_{\theta_{i}}$
is a density admitted by the marginal distribution $P\left(T_{i}\in\bullet\right)$
for all $i\in\left\{ 1,\dots,N\right\} $. With suitable weights and
the assumption that $\hat{\theta}_{i}\left(\mathbf{T}\right)=\arg\sup_{\theta\in\Theta}\bar{L}_{i}\left(\theta;\mathbf{T}\right)$
is almost surely unique for all $i\in\left\{ 1,\dots,N\right\} $,
the difference between $\hat{\theta}_{i}\left(\mathbf{T}\right)$
and the conventional maximum likelihood estimator of $\theta_{i}$
almost surely converges to 0 as $n_{i}$ diverges with $N$ held fixed.
Specifically, $w_{ii}=1+o_{P}\left(1\right)$ and $i\ne j\implies w_{ij}=o_{P}\left(1\right)$
ensure that $\hat{\theta}_{i}\left(\mathbf{T}\right)=\arg\sup_{\theta\in\Theta}L\left(\theta;x_{i}\right)+o_{P}\left(1\right)$,
where the term $o_{P}\left(1\right)$ converges to 0 with $P$-probability
1 as $n_{i}\rightarrow\infty$ with any ratio $n_{j}/n_{k}$ bounded
by constants: $n_{j}=O\left(n_{k}\right)$ for all $j,k\in\left\{ 1,\dots,N\right\} $.

\subsubsection{Predictive loss}

For some $\bar{g}\in\mathcal{E}\left(\mathcal{T}_{i}\right)$, the
\emph{generalized regret\begin{equation}
\reg_{i}\left(\bar{g},\mathbf{t};\Theta\right)=-\log\bar{g}\left(\tau\left(x_{i}\right)\right)-\inf_{\theta\in\Theta}\left(-\log\bar{L}_{i}\left(\theta;\mathbf{t}\right)\right)=\log\frac{\bar{L}_{i}\left(\hat{\theta}_{i}\left(\mathbf{t}\right);\mathbf{t}\right)}{\bar{g}\left(\tau\left(x_{i}\right)\right)}\label{eq:generalized-regret}\end{equation}
}measures loss incurred by the likelihood associated with $\bar{g}$,
the predictive distribution, relative to $\bar{L}_{i}\left(\hat{\theta}_{i}\left(\mathbf{t}\right);\mathbf{t}\right)$,
the maximum weighted likelihood of $\theta_{i}$. In other words,
$\reg_{i}\left(\bar{g},\mathbf{t};\Theta\right)$ is the discrepancy
between error in predicting the value of $\tau\left(x\right)$ on
the basis of $\bar{g}$ and the prediction error minimized over the
interest parameter. The latter error is more relevant to hypotheses
about the value of $\theta$ than a prediction error minimized over
the full parameter $\phi$, including the nuisance parameter $\lambda$
(§\ref{sub:Optimality-observed}). Thus, $\reg_{i}\left(\bar{g},\mathbf{t};\Theta\right)$
replaces $\reg\left(\bar{f},x;\Phi\right)$ as the regret in the presence
of the nuisance parameter or a nonzero weight other than $w_{ii}$.

\subsection{\label{sub:Optimal-predictive-distribution}Optimal predictive distribution}

\subsubsection{\label{sub:NMWL}Exact predictive distribution}

For each $t\in\mathcal{T}_{i}$, let $\mathbf{t}_{i}\left(t\right)$
denote the $N$-tuple of statistics that is equal to $\mathbf{t}$
in all components except the $i$th, which has $t$ in place of $t_{i}$.
For example, $\mathbf{t}_{1}\left(t\right)=\left\langle t,t_{2},\dots,t_{N}\right\rangle $,
but $\mathbf{t}_{i}\left(t\right)=\left\langle t_{1},\dots,t_{i-1},t,t_{i+1},\dots,t_{N}\right\rangle $
if $3\le i\le N-2$. 

The optimal predictive density function of Section \ref{sub:Optimality-observed}
is a special case of \[
\bar{g}_{i}=\arg\inf_{\bar{g}\in\mathcal{E}\left(\mathcal{T}_{i}\right)}\sup_{t\in\mathcal{T}_{i}}\reg_{i}\left(\bar{g},\mathbf{t}_{i}\left(t\right);\Theta\right),\]
the $\mathcal{E}\left(\mathcal{T}_{i}\right)$-\emph{optimal predictive
density function relative to }$\left\langle \mathcal{G},w_{i}\right\rangle $.
\begin{thm}
\label{thm:Shtarkov}Given some $i\in\left\{ 1,\dots,N\right\} $
and $\mathbf{t}\in\mathcal{T}_{1}\times\cdots\times\mathcal{T}_{N}$,
if $\int\bar{L}_{i}\left(\hat{\theta}_{i}\left(\mathbf{t}_{i}\left(t\right)\right);\mathbf{t}_{i}\left(t\right)\right)dt<\infty$,
then, for all $t_{i}\in\mathcal{T}_{i}$, the \textup{$\mathcal{E}\left(\mathcal{T}_{i}\right)$}-optimal
predictive density function relative to \textup{$\left\langle \mathcal{G},w_{i}\right\rangle $}
satisfies\begin{equation}
\bar{g}_{i}\left(t_{i}\right)=\frac{\bar{L}_{i}\left(\hat{\theta}_{i}\left(\mathbf{t}\right);\mathbf{t}\right)}{\int_{\mathcal{T}_{i}}\bar{L}_{i}\left(\hat{\theta}_{i}\left(\mathbf{t}_{i}\left(t\right)\right);\mathbf{t}_{i}\left(t\right)\right)dt}.\label{eq:NMWL}\end{equation}
\end{thm}
\begin{proof}
The present argument follows that used to prove Theorem \ref{thm:NML}.
Assume, contrary to the claim, that the density function $\bar{g}_{i}$
that satisfies equation \eqref{eq:NMWL} for all $t_{i}\in\mathcal{T}_{i}$
is not the optimal predictive density function relative to $\left\langle \mathcal{G},w_{i}\right\rangle $.
The substitution $\bar{L}_{i}\left(\hat{\theta}_{i}\left(\mathbf{t}\right);\mathbf{t}\right)=\bar{L}_{i}\left(\hat{\theta}_{i}\left(\mathbf{t}_{i}\left(t_{i}\right)\right);\mathbf{t}_{i}\left(t_{i}\right)\right)$
demonstrates that the ratio $\bar{g}_{i}\left(t_{i}\right)/\bar{L}_{i}\left(\hat{\theta}_{i}\left(\mathbf{t}_{i}\left(t_{i}\right)\right);\mathbf{t}_{i}\left(t_{i}\right)\right)$
does not depend on $t_{i}$. It follows that, for any $\breve{g}_{i}\in\mathcal{E}\left(\mathcal{T}_{i}\right)\backslash\left\{ \bar{g}_{i}\right\} $,
there is a $t_{i}\in\mathcal{T}_{i}$ such that $\breve{g}_{i}\left(t_{i}\right)/\bar{L}_{i}\left(\hat{\theta}_{i}\left(\mathbf{t}_{i}\left(t_{i}\right)\right);\mathbf{t}_{i}\left(t_{i}\right)\right)<\bar{g}_{i}\left(t_{i}\right)/\bar{L}_{i}\left(\hat{\theta}_{i}\left(\mathbf{t}_{i}\left(t_{i}\right)\right);\mathbf{t}_{i}\left(t_{i}\right)\right)$.
Therefore, given any $\breve{g}_{i}\in\mathcal{E}\left(\mathcal{T}_{i}\right)\backslash\left\{ \bar{g}_{i}\right\} $,
there is a $t_{i}\in\mathcal{T}_{i}$ such that $\reg_{i}\left(\bar{g}_{i},\mathbf{t}_{i}\left(t_{i}\right);\Theta\right)<\reg_{i}\left(\breve{g}_{i},\mathbf{t}_{i}\left(t_{i}\right);\Theta\right)$,
which contradicts the assumption. 
\end{proof}
For any $x_{i}\in\mathcal{X}^{n}$, the quantity $\bar{g}_{i}\left(\tau\left(x_{i}\right)\right)=\bar{g}_{i}\left(\tau\left(x_{i}\right);\Theta\right)$
is the \emph{normalized maximum weighted likelihood (NMWL)} with respect
to $\Theta$ or, more precisely, with respect to $\left\langle \mathcal{G},w_{i}\right\rangle $.
\begin{example}
\label{exa:conditional-NML}When the constraint that $\sum_{j=1}^{N}w_{ij}=1$
is relaxed, NMWL generalizes various previous NMLs as follows. If
$\tau\left(x_{1}\right)=x_{1}$ for some $x_{1}\in\mathcal{X}^{n_{1}}$
and if $N=1$, the NMWL reduces to the probability density $\bar{g}_{1}\left(x_{1}\right)$
with $w_{1,1}=1$ and thus to $\bar{f}_{0}\left(x\right)$, the NML
of equation \eqref{eq:NML}. For an observed vector $\left\langle y_{1},\dots,y_{n_{1}}\right\rangle \in\mathcal{X}^{n_{1}}$,
assigning $N=2$, $\theta_{1}=\theta_{2}$, $t_{1}=\left\langle y_{1},\dots,y_{n_{1}-1}\right\rangle $,
and $t_{2}=y_{n_{1}}$ demonstrates that the prominent conditional
NMLs are NMWLs in the case of IID data. In particular, \citet[§11.4.2]{RefWorks:375}
considered $\bar{g}_{1}\left(\left\langle t_{1},t_{2}\right\rangle \right)$
with $w_{1,1}=w_{1,2}=1$. Conversely, \citet[§5.2.4]{RefWorks:374}
and \citet{RissanenRoos2007b} studied $\bar{g}_{2}\left(\left\langle t_{1},t_{2}\right\rangle \right)$
with $w_{2,1}=w_{2,2}=1$, thereby facilitating computation of the
normalizing constant in equation \eqref{eq:NMWL} since the integration
is only over a scalar. The main drawback of applying conditional NMLs
to the IID setting is the arbitrary nature of choosing an observation
$x_{2}$ to leave out since the observations are not ordered in time
\citep[§11.4.3]{RefWorks:375}. The same issue arises in Bayesian
model selection when an improper prior is conditioned on a minimal
training sample before computing the Bayes factor. A popular solution
is to take geometric or arithmetic averages over all possible minimal
training samples \citep{Berger2004841}. Analogous approaches to IID
applications of conditional NMLs would likewise depend on arbitrary
choices of averages and of training sample sizes (§\ref{sub:Quantifying-statistical-evidence}).
\end{example}

\subsubsection{Approximate predictive distribution}

A computationally efficient approximation to the NMWL is available
if:
\begin{enumerate}
\item The weight of any comparison in focus is equal to that of any other
comparison when it is in focus, i.e., $w_{ii}=w_{1,1}$ for all $i$.
\item The weight of each comparison not in focus is equal to that of any
other comparison not in focus, i.e., $w_{ij}=w_{1,2}$ for all $i\ne j$.
\item The sample sizes and sample spaces are equal, i.e., $n_{i}=n_{1}$
and $\mathcal{T}_{i}=\mathcal{T}_{1}$ for all $i$.
\item All comparisons share a single family, i.e., $\mathcal{G}_{n_{1},i}=\mathcal{G}_{n_{1},1}$
and $L_{1}=L_{i}$ for all $i$.
\end{enumerate}
Under those \emph{equal weight conditions}, there is an approximate
weight $\tilde{w}_{N+1}$ such that $\tilde{w}_{N+1}=w_{ii}$ for
all $i$ and an approximate weight $\tilde{w}_{j}$ such that $\tilde{w}_{j}=\left(N-1\right)N^{-1}w_{ij}$
for all $i$ and $j$ except $i=j$. Then $\sum_{j=1}^{N+1}\tilde{w}_{j}=N\tilde{w}_{1}+\tilde{w}_{N+1}=\left(N-1\right)w_{1,2}+w_{1,1}=1$.

For any $t\in\mathcal{T}_{1}$, let $\tilde{\mathbf{t}}\left(t\right)$
denote $\left\langle t_{1},\dots,t_{N},t\right\rangle \in\mathcal{T}_{1}^{N+1}$.
For inference about $\theta_{i}$ on the basis of $\mathbf{t}$, the
\emph{approximate weighted likelihood function} $\tilde{L}\left(\bullet;\tilde{\mathbf{t}}\left(t\right)\right):\Theta\rightarrow\left[0,\infty\right)$
is defined by \begin{eqnarray*}
\log\tilde{L}\left(\theta_{i};\tilde{\mathbf{t}}\left(t\right)\right) & = & \sum_{j=1}^{N}\tilde{w}_{j}\log L_{j}\left(\theta_{i};t_{j}\right)+\tilde{w}_{N+1}\log L_{j}\left(\theta_{i};t\right)\\
 & = & \frac{1-w_{1,1}}{N}\sum_{j=1}^{N}\log L_{1}\left(\theta_{i};t_{j}\right)+w_{1,1}\log L_{1}\left(\theta_{i};t\right),\end{eqnarray*}
the second equality implied by the equal weight conditions. Let $\hat{\theta}\left(\tilde{\mathbf{t}}\left(t\right)\right)=\arg\sup_{\theta\in\Theta}\tilde{L}\left(\theta;\tilde{\mathbf{t}}\left(t\right)\right)$. 

The following theorem indicates that the exact NMWL \eqref{eq:NMWL}
is approximated by\[
\tilde{g}_{i}\left(t_{i}\right)=\frac{\bar{L}_{i}\left(\hat{\theta}_{i}\left(\mathbf{t}\right);\mathbf{t}\right)}{\int_{\mathcal{T}_{1}}\tilde{L}\left(\hat{\theta}\left(\tilde{\mathbf{t}}\left(t\right)\right);\tilde{\mathbf{t}}\left(t\right)\right)dt},\]
which may be quickly calculated even for large $N$ since the denominator,
not depending on $i$, need only be computed once. In the theorem
and its supporting lemmas, $\tilde{\mathbf{T}}\left(t\right)=\left\langle T_{1},\dots,T_{N},t\right\rangle $,
and $\overset{\text{a.s.}}{\rightarrow}$ denotes almost sure convergence
as $N$ increases with $n_{1}$ fixed.
\begin{lem}
\label{lem:approximation}If the equal weight conditions hold and
if $T_{1},\dots,T_{N}$ are drawn independently from a mixture distribution,
then, for all $\theta\in\Theta$ and $t\in\mathcal{T}_{i}$,\begin{equation}
\log\tilde{L}\left(\theta;\tilde{\mathbf{T}}\left(t\right)\right)-\log\bar{L}_{i}\left(\theta;\mathbf{T}_{i}\left(t\right)\right)\overset{\text{a.s.}}{\rightarrow}0.\label{eq:approximation}\end{equation}
\end{lem}
\begin{proof}
According to the equal weight conditions,\begin{eqnarray*}
\log\tilde{L}\left(\theta;\tilde{\mathbf{T}}\left(t\right)\right)-\log\bar{L}_{i}\left(\theta;\mathbf{T}_{i}\left(t\right)\right) & =\end{eqnarray*}
\[
\left(\frac{1-\tilde{w}_{N+1}}{N}\sum_{j=1}^{N}\log L_{1}\left(\theta;T_{j}\right)+\tilde{w}_{N+1}\log L_{1}\left(\theta;T_{i}\right)\right)-\left(\frac{1-w_{ii}}{N-1}\sum_{j\ne i;i=1}^{N}\log L_{1}\left(\theta;T_{j}\right)+w_{ii}\log L_{1}\left(\theta;T_{i}\right)\right)\]
\[
=\left(1-\tilde{w}_{N+1}\right)\left(\frac{1}{N}\sum_{j=1}^{N}\log L_{1}\left(\theta;T_{j}\right)-\frac{1}{N-1}\sum_{j\ne i;i=1}^{N}\log L_{1}\left(\theta;T_{j}\right)\right).\]
The second factor almost surely vanishes by the law of large numbers.\end{proof}
\begin{lem}
\label{lem:MLEs}Under the assumptions of Lemma \ref{lem:approximation},
the stipulations that $\hat{\theta}\left(\tilde{\mathbf{T}}\left(t\right)\right)$
and $\hat{\theta}_{i}\left(\mathbf{T}_{i}\left(t\right)\right)$ are
almost always unique for all $t\in\mathcal{T}_{i}$ and that $L_{1}\left(\bullet;T_{i}\right)$
is almost surely continuous on $\Theta$ for all $i=1,\dots,N$ imply
that, for all $i=1,\dots,N$ and $t\in\mathcal{T}_{i}$, $\hat{\theta}\left(\tilde{\mathbf{T}}\left(t\right)\right)-\hat{\theta}_{i}\left(\mathbf{T}_{i}\left(t\right)\right)\overset{\text{a.s.}}{\rightarrow}0.$\end{lem}
\begin{proof}
By Lemma \ref{lem:approximation}, equation \eqref{eq:approximation}
holds for all $\theta\in\Theta$. Thus, since almost sure convergence
is preserved under almost surely continuous transformations \citep[§1.7]{Serfling:1254212},
\\
$\arg\sup_{\theta\in\Theta}\tilde{L}\left(\theta;\tilde{\mathbf{T}}\left(t\right)\right)-\arg\sup_{\theta\in\Theta}\bar{L}_{i}\left(\theta;\mathbf{T}\right)\overset{\text{a.s.}}{\rightarrow}0.$\end{proof}
\begin{thm}
\label{thm:approximate-complexity}Under the assumptions of Lemma
\ref{lem:MLEs}, the difference between the approximate and exact
parametric complexities almost surely vanishes: \begin{equation}
\int_{\mathcal{T}_{1}}\tilde{L}\left(\hat{\theta}\left(\tilde{\mathbf{T}}\left(t\right)\right);\tilde{\mathbf{T}}\left(t\right)\right)dt-\int_{\mathcal{T}_{i}}\bar{L}_{i}\left(\hat{\theta}_{i}\left(\mathbf{T}_{i}\left(t\right)\right);\mathbf{T}_{i}\left(t\right)\right)dt\overset{\text{a.s.}}{\rightarrow}0.\label{eq:approximate-complexity}\end{equation}
\end{thm}
\begin{proof}
Combining the results of Lemmas \ref{lem:approximation} and \ref{lem:MLEs}
gives\[
\tilde{L}\left(\hat{\theta}\left(\tilde{\mathbf{T}}\left(t\right)\right);\tilde{\mathbf{T}}\left(t\right)\right)-\bar{L}_{i}\left(\hat{\theta}_{i}\left(\mathbf{T}_{i}\left(t\right)\right);\mathbf{T}_{i}\left(t\right)\right)\overset{\text{a.s.}}{\rightarrow}0\]
for all $t\in\mathcal{T}_{i}$ since $\mathcal{T}_{i}=\mathcal{T}_{1}$
and the functions are almost surely continuous by assumption.
\end{proof}

\subsection{\label{sub:Optimal-discrimination-information}Optimal discrimination
information}

For any $\Theta^{\prime}\subseteq\Theta$, let $\bar{g}_{i}\left(t_{i};\Theta^{\prime}\right)$
denote the optimal predictive density function relative to $\left\langle \left\{ g_{\theta}:\theta\in\Theta^{\prime}\right\} ,w_{i}\right\rangle $
as defined in Section \ref{sub:NMWL}. For any $\Theta_{0},\Theta_{1}\subseteq\Theta$,
the \emph{optimal information in $x$ for discrimination} in favor
of the hypothesis that $\theta_{i}\in\Theta_{1}$ over the hypothesis
that $\theta_{i}\in\Theta_{0}$ is\[
\bar{I}_{i}\left(\Theta_{1},\Theta_{0}\right)=-\log\bar{g}_{i}\left(t_{i};\Theta_{0}\right)-\left(-\log\bar{g}_{i}\left(t_{i};\Theta_{1}\right)\right),\]
generalizing quantities in \citet{RefWorks:292}, \citet{RefWorks:342}, \citet{RefWorks:435}, and \citet{mediumScale}.
The approximate optimal information $\tilde{I}_{i}\left(\Theta_{1},\Theta_{0}\right)$
is defined identically except with $\tilde{g}_{i}$ in place of $\bar{g}_{i}$.
$\bar{I}_{i}\left(\Theta_{1},\Theta_{0}\right)$ is not restricted
to the case of smoothness conditions on $\left\{ g_{\theta}:\theta\in\Theta\right\} $,
but applies to any problem of selecting one of two models.

Since $\bar{g}_{i}\left(t_{i};\Theta_{1}\right)/\bar{g}_{i}\left(t_{i};\left\{ \theta_{0}\right\} \right)$
for $\theta_{0}\in\Theta$is a likelihood ratio, the discrimination
information has the universal bound on the probability of misleading
evidence under $\theta=\theta_{0}$ (\citealp{RefWorks:123}; \citealp{mediumScale}).
The next lemma and theorem bear on whether the optimal information
for discrimination is an interpretable measure of evidence in that
the probability of observing misleading information converges to 0
as $n_{i}\rightarrow\infty$. Let $\hat{\theta}_{i}\left(\mathbf{T};\Theta^{\prime}\right)=\arg\sup_{\theta\in\Theta^{\prime}}\bar{L}_{i}\left(\theta;\mathbf{T}\right)$
for $i=1,\dots,N$ and $\hat{\theta}_{0}\left(t;\Theta^{\prime}\right)=\arg\sup_{\theta\in\Theta^{\prime}}L_{i}\left(\theta;t\right)$
given any $\Theta^{\prime}\subseteq\Theta$.
\begin{lem}
\label{lem:complexity}Suppose $\Theta=\mathbb{R}^{D}$, $\theta=\left\langle \theta_{1},\dots,\theta_{D}\right\rangle ^{\T}$,
$n_{j}=O\left(n_{k}\right)$ for all $j,k\in\left\{ 1,\dots,N\right\} $,
$\mathcal{G}_{n,i}=\mathcal{G}_{n,1}$ and $L_{i}=L_{1}$ for all
$i\in\left\{ 1,\dots,N\right\} $ and sufficiently large $n$, and
$\tau\left(x\right)=x$ for all $x\in\mathcal{X}^{n}$, which implies
that $\mathcal{T}_{i}=\mathcal{X}^{n}$, $t_{i}=x_{i}$, and $T_{i}=X_{i}$.
Assume also that for some $i\in\left\{ 1,\dots,N\right\} $, there
exists an open, bounded set $\Theta^{\prime}\subseteq\Theta$ on which
$L\left(\bullet;X_{i}\right)$ is almost surely continuous and such
that \begin{equation}
\log\int_{\mathcal{X}^{n}}\bar{L}_{i}\left(\hat{\theta}_{0}\left(t;\Theta^{\prime}\right);t\right)dt=\frac{D}{2}\log\frac{n_{i}}{2\pi}+\log\int_{\Theta^{\prime}}\sqrt{\frac{1}{n_{i}}\left|E\frac{\partial^{2}\ln g_{\theta}\left(X_{i}\right)}{\partial\theta\partial\theta^{\T}}\right|}d\theta+o\left(1\right).\label{eq:complexity}\end{equation}
\[
\therefore\log\int_{\mathcal{T}_{i}}\bar{L}_{i}\left(\hat{\theta}_{i}\left(\mathbf{T}_{i}\left(t\right);\Theta^{\prime}\right);\mathbf{T}_{i}\left(t\right)\right)dt=\frac{D}{2}\log\frac{n_{i}}{2\pi}+\log\int_{\Theta^{\prime}}\sqrt{\frac{1}{n_{i}}\left|E\frac{\partial^{2}\ln g_{\theta}\left(X_{i}\right)}{\partial\theta\partial\theta^{\T}}\right|}d\theta+o\left(1\right)\]
almost surely holds for any weights that satisfy $P\left(\lim_{n_{i}\rightarrow\infty}w_{ii}=1\right)=1$
and $\sum_{j=1}^{N}w_{ij}=1$.\end{lem}
\begin{proof}
The continuity condition and the constraints on the weights and sample
sizes ensure that $\bar{L}_{i}\left(\hat{\theta}_{i}\left(\mathbf{T}_{i}\left(t\right);\Theta^{\prime}\right);\mathbf{T}_{i}\left(t\right)\right)\overset{\text{a.s.}}{\rightarrow}\bar{L}_{i}\left(\hat{\theta}_{0}\left(t;\Theta^{\prime}\right);t\right)$
as $n_{i}\rightarrow\infty$ for all $t\in\mathcal{T}_{i}$.
\end{proof}
The assumptions of Lemma \ref{lem:complexity} are broadly applicable
since equation \eqref{eq:complexity} holds under general regularity
conditions \citep{RefWorks:396}. The result will now be extended
to non-bounded parameter spaces.
\begin{thm}
\label{thm:unbounded-complexity}Suppose that $\Theta_{1}\subseteq\Theta$,
that $n_{j}=O\left(n_{k}\right)$ for all $j,k\in\left\{ 1,\dots,N\right\} $,
and that $P$ is the sampling distribution of $\mathbf{T}$. Assume
also that for any $i\in\left\{ 1,\dots,N\right\} $, there exists
an open, bounded set $\Theta^{\prime}\subseteq\Theta_{1}$ such that
\begin{equation}
P\left(\lim_{n_{i}\rightarrow\infty}\log\int_{\mathcal{T}_{i}}\bar{L}_{i}\left(\hat{\theta}_{i}\left(\mathbf{T}_{i}\left(t\right);\Theta^{\prime}\right);\mathbf{T}_{i}\left(t\right)\right)dt=\infty\right)=1.\label{eq:unbounded-complexity}\end{equation}
\[
\therefore P\left(\lim_{n_{i}\rightarrow\infty}\log\int_{\mathcal{T}_{i}}\bar{L}_{i}\left(\hat{\theta}_{i}\left(\mathbf{T}_{i}\left(t\right);\Theta_{1}\right);\mathbf{T}_{i}\left(t\right)\right)dt=\infty\right)=1.\]
\end{thm}
\begin{proof}
Let $\mathfrak{T}=\left\{ t\in\mathcal{T}_{i}:\hat{\theta}_{i}\left(\mathbf{T}_{i}\left(t\right);\Theta_{1}\right)\in\Theta^{\prime}\right\} $
to expand $\int_{\mathcal{T}_{i}}\bar{L}_{i}\left(\hat{\theta}_{i}\left(\mathbf{T}_{i}\left(t\right);\Theta_{1}\right);\mathbf{T}_{i}\left(t\right)\right)dt$
as \[
\int_{\mathfrak{T}}\bar{L}_{i}\left(\hat{\theta}_{i}\left(\mathbf{T}_{i}\left(t\right);\Theta_{1}\right);\mathbf{T}_{i}\left(t\right)\right)dt+\int_{\mathcal{T}_{i}\backslash\mathfrak{T}}\bar{L}_{i}\left(\hat{\theta}_{i}\left(\mathbf{T}_{i}\left(t\right);\Theta_{1}\right);\mathbf{T}_{i}\left(t\right)\right)dt.\]
Thus, since $\bar{L}_{i}\left(\hat{\theta}_{i}\left(\mathbf{T}_{i}\left(t\right);\Theta_{1}\right);\mathbf{T}_{i}\left(t\right)\right)=\bar{L}_{i}\left(\hat{\theta}_{i}\left(\mathbf{T}_{i}\left(t\right);\Theta^{\prime}\right);\mathbf{T}_{i}\left(t\right)\right)$
for all $t\in\mathfrak{T}$ and \[
\bar{L}_{i}\left(\hat{\theta}_{i}\left(\mathbf{T}_{i}\left(t\right);\Theta_{1}\right);\mathbf{T}_{i}\left(t\right)\right)>\bar{L}_{i}\left(\hat{\theta}_{i}\left(\mathbf{T}_{i}\left(t\right);\Theta^{\prime}\right);\mathbf{T}_{i}\left(t\right)\right)\]
for all $t$ in non-empty $\mathcal{T}_{i}\backslash\mathfrak{T}$
given any sufficiently large $n_{i}$,\[
\int_{\mathcal{T}_{i}}\bar{L}_{i}\left(\hat{\theta}_{i}\left(\mathbf{T}_{i}\left(t\right);\Theta_{1}\right);\mathbf{T}_{i}\left(t\right)\right)dt\ge\int_{\mathcal{T}_{i}}\bar{L}_{i}\left(\hat{\theta}_{i}\left(\mathbf{T}_{i}\left(t\right);\Theta^{\prime}\right);\mathbf{T}_{i}\left(t\right)\right)dt\]
follows, where the equality and both inequalities hold with $P$-probability
1.
\end{proof}
Since the claim of Lemma \ref{lem:complexity} implies equation \eqref{eq:unbounded-complexity},
Theorem \ref{thm:unbounded-complexity} applies to the wide class
of models satisfying the regularity conditions of \citet{RefWorks:396}.
The largely overlapping regularity conditions of \citet{RefWorks:387}
then ensure that $\lim_{n_{i}\rightarrow\infty}P\left(\bar{I}_{i}\left(\Theta_{1},\left\{ \theta_{0}\right\} \right)>0\right)=0$
when there is no $\theta\in\Theta$ such that $g_{\theta}$ is closer
in Kullback-Leibler divergence than $g_{\theta_{0}}$ to the marginal
distribution $P\left(T_{i}\in\bullet\right)$. In the special case
of correct model specification considered in \citet{RefWorks:123} and \citet{mediumScale},
the equation holds for all $P$ admitting $g_{\theta_{0}}$ as the
marginal density of $T_{i}$.

\subsection{\label{sub:Single-observation-weights}Single-observation weights}

This section defines \emph{single-observation weights} as the components
of $w_{i}$ such that for every $i\in1,\dots,N$ that all incidental
data (all $x_{j}$ with $j\ne i$) together have the weight of one
observation in the focus vector $x_{i}$ $\left(\sum_{j\ne i}w_{ij}=w_{ii}/n_{i}\right)$
and that each comparison other than the $i$th has equal weight $\left(\forall j,k\ne i\, w_{ij}=w_{ik}\right)$.
Solving those equations and $\sum_{j=1}^{N}w_{ij}=1$ uniquely gives
$w_{ii}=1-\left(n_{i}+1\right)^{-1}$ and $i\ne j\implies w_{ij}=\left(n_{i}+1\right)^{-1}\left(N-1\right)^{-1}$. 

If there is only a single comparison, then its observed statistic
$t_{1}=\tau\left(x_{1}\right)$ is supplemented by a pseudo-statistic
$t_{0}$, a scientifically meaningful value in $\mathcal{T}_{1}$
that does not depend on $x_{1}$. For example, $t_{0}$ might be $\int tg_{\theta_{0}}\left(t\right)dt$,
the expectation value of $T_{1}$ under $\theta=\theta_{0}$. (Similarly,
\citet{RefWorks:1147} considered the use of a prior with the Fisher
information of a single pseudo-observation.) The use of $\mathbf{t}=\left\langle t_{0},t_{1}\right\rangle $
and $N=2$ with single-observation weights then entails that \begin{equation}
\log\bar{L}_{1}\left(\theta_{1};\mathbf{t}\right)=\left(n_{1}+1\right)^{-1}\log L_{1}\left(\theta_{1};t_{0}\right)+\left(1-\left(n_{1}+1\right)^{-1}\right)\log L_{1}\left(\theta_{1};t_{1}\right).\label{eq:null-weights}\end{equation}
For a smoother transition from a single comparison to multiple comparisons,
the pseudo-statistic may be assigned the same weight as each of the
$N-1$ incidental statistics among $t_{1},\dots,t_{N}$, i.e., $w_{ij}=\left(n_{i}+1\right)^{-1}N^{-1}$
for all $j\in\left\{ 0,1,\dots,N\right\} \backslash\left\{ i\right\} $.

The following result applies whether there is a single comparison
or multiple comparisons.
\begin{cor}
Assume the components of $w_{i}$ are single-observation weights,
that $n_{i}=n_{1}$ for all $i=1,\dots,N$, and that $\mathcal{G}_{n_{1},i}=\mathcal{G}_{n_{1},1}$
for all $i=1,\dots,N$. If $T_{1},\dots,T_{N}$ are independent and
drawn from a mixture distribution, if $\hat{\theta}\left(\tilde{\mathbf{T}}\left(t\right)\right)$
and $\hat{\theta}_{i}\left(\mathbf{T}_{i}\left(t\right)\right)$ are
almost always unique for all $t\in\mathcal{T}_{i}$, and if $L_{i}\left(\bullet;T_{i}\right)$
is almost surely continuous on $\Theta$ for all $i=1,\dots,N$, then
equation \eqref{eq:approximate-complexity} holds.\end{cor}
\begin{proof}
All the conditions of Theorem \ref{thm:approximate-complexity} are
given except for the equal weights condition, which follows from the
single-observation weights assumption, the equality of the sample
sizes, and the commonality of the family of distributions.
\end{proof}

\section{\label{sec:Case-study}Case studies}

In the following models, $\int f_{\hat{\phi}\left(x\right)}\left(x\right)dx=\infty$,
rendering the unweighted NML \eqref{eq:NML} useless. The NMWL \eqref{eq:NMWL}
can be used instead since $\int_{\mathcal{T}_{i}}\bar{L}_{i}\left(\hat{\theta}_{i}\left(\mathbf{t}_{i}\left(t\right)\right);\mathbf{t}_{i}\left(t\right)\right)dt<\infty$.

Results of two separate NMWL analyses are presented for each application.
The first uses multiple comparisons for inference relevant to each
comparison \eqref{eq:weighted-likelihood}. The second uses $\int tg_{\theta_{0}}\left(t\right)dt=0$
in place of data associated with other comparisons, as if there were
only a single comparison \eqref{eq:null-weights}. All plots use the
binary logarithm to express information in bits and display a different
value for each comparison.

\subsection{\label{sub:Populations}Single and multiple populations}

Before addressing a problem in contemporary biology, the proposed
methodology will be illustrated using a simple data set that has motivated
both Bayesian \citep{Rubin1981b} and weighted likelihood \citep{Wang2006279}
approaches. The reduced data consist of the estimated average effect
of a training program on SAT scores and an estimated standard error
of the effect estimate for each of eight test sites. Following the
tradition continued by \citet{Wang2006279}, the standard errors $\sigma_{1},\dots,\sigma_{8}$
are considered known, and the effect estimates are modeled as normal
observations with unknown means $\theta_{1},\dots,\theta_{8}$. Thus,
$N=8$ and $\left\{ g_{i,\theta_{i}}:\theta\in\Theta\right\} $ is
the family of distributions, where $g_{i,\theta_{i}}$ is the normal
density of mean $\theta_{i}$ and standard deviation $\sigma_{i}$.
For the $i$th site, $\theta_{i}\ne0$ is the alternative hypothesis
and $\theta_{i}=0$ is the null hypothesis.

Fig. \ref{fig:infoSAT} displays $\tilde{I}_{i}\left(\mathbb{R}\backslash\left\{ 0\right\} ,\left\{ 0\right\} \right)$,
the resulting approximate discrimination information, with $\bar{I}_{i}\left(\mathbb{R}\backslash\left\{ 0\right\} ,\left\{ 0\right\} \right)$,
the exact discrimination information. As in Section \ref{sub:Single-observation-weights},
the weight of a single observation is assigned either to 0, the null
hypothesis value ({}``information from null''), or to the incidental
testing sites ({}``information from sites''). The resulting information
values are barely distinguishable. %
\begin{figure}
\includegraphics[scale=0.8]{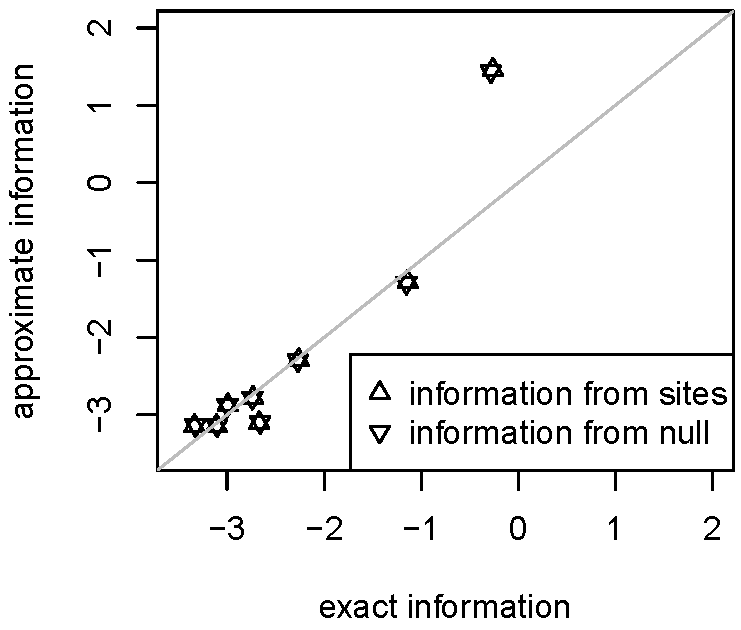}

\caption{\label{fig:infoSAT}Information (bits) favoring the hypothesis that
the test score at a site was affected by the treatment.}

\end{figure}

\subsection{\label{sub:Biological-features}Single and multiple biological features}

In typical experiments measuring gene expression or the abundance
of proteins or metabolites, the primary question is whether the expectation
value of a logarithm of the expression or abundance of each feature
is affected by a treatment, disease, or other perturbation. Since
that question is equivalent to that of whether $\text{CV}_{i}^{-1}$,
the inverse coefficient of variation\emph{ }for the $i$th feature,
is 0, the data reduction strategy of Example \ref{exa:microarray}
often proves effective even if the magnitude of $\text{CV}_{i}^{-1}$
is not of direct interest. $\text{CV}_{i}^{-1}$ has a one-to-one
correspondence to the proportion of the feature-feature pairs with
abundance ratios greater than 1 \citep{RefWorks:21,RefWorks:26}.
In addition, $\text{CV}_{i}^{-1}$ is often of more scientific interest
than the mean since small changes in numbers of biomolecules can have
a strong influence on downstream processes. 

The method of Example \ref{exa:microarray} is applied to the proteomics
data set of Alex Miron\textquoteright{}s lab at the Dana-Farber Cancer
Institute \citep{ProData2009b}, with $x_{ij}$ and $y_{ij}$ as the
logarithms of the abundance levels of the $i$th of $N=20$ proteins
in the $j$th woman with and without breast cancer, respectively,
after the preprocessing of \citet{mediumScale}. Likewise, $\xi_{i}$
and $\eta_{i}$ are the expectation values of the random variables
$X_{ij}$ and $Y_{ij}$. Each of two breast cancer groups (one of
55 HER2-positive women and the other of 35 women mostly-ER/PR-positive)
were compared to a control group of 64 women. Since $\theta_{i}=\left|\text{CV}_{i}^{-1}\right|$
and thus $\Theta=\left[0,\infty\right)$, the competing hypotheses
for the $i$th protein are $\theta_{i}>0$ and $\theta_{i}=0$.

The left panel of Fig. \ref{fig:infoAB} displays the approximate
information for discrimination in favor of the alternative hypothesis
that $\theta_{i}\ne0$ over the null hypothesis that $\theta_{i}=0$
by weighing the incidental proteins as a single observation (§\ref{sub:Single-observation-weights}).
$\tilde{I}_{i}\left(\left(0,\infty\right),\left\{ 0\right\} \right)$,
the approximate optimal information, is compared to $\log\left(g_{i}\left(\hat{\theta}_{\text{MLE}};t_{i}\right)/g_{i}\left(0;t_{i}\right)\right)$.
Here, $\hat{\theta}_{\text{MLE}}$ is common to all proteins, denoting
the maximum likelihood estimate (MLE) defined under the assumptions
that $\theta_{i}\in\left\{ 0,\theta_{\text{alt.}}\right\} $ for some
$\theta_{\text{alt.}}>0$ for all $i$ and that the test statistics
are independent \citep{mediumScale}. The right panel of Fig. \ref{fig:infoAB}
contrasts the widely varying regret of the MLE information with the
constant regret of the optimal information.

\begin{figure}
\includegraphics[scale=0.8]{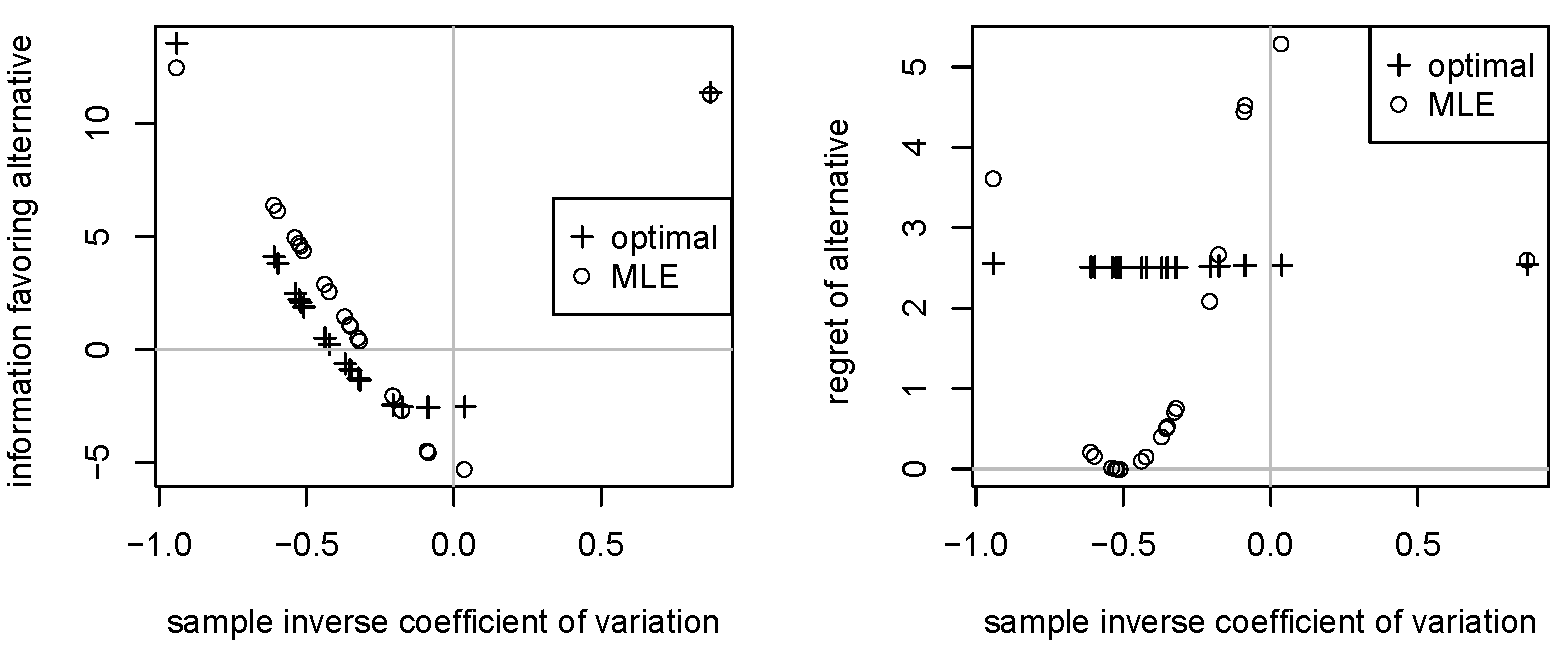}

\caption{\label{fig:infoAB}\emph{Left panel:} Discrimination information (bits)
favoring the hypothesis that the abundance level of a protein differs
by disease status versus $\widehat{\text{CV}}_{i}^{-1}$. \emph{Right
panel:} The corresponding regret versus $\widehat{\text{CV}}_{i}^{-1}$.
($\widehat{\text{CV}}_{i}^{-1}$ denotes the difference in sample
means divided by the sample standard deviation for the $i$th protein.)}

\end{figure}

Giving the null hypothesis the weight of a single observation \eqref{eq:null-weights},
as if the abundance level of only one protein were measured, results
in information values that are visually indistinguishable from those
of Fig. \ref{fig:infoAB}. Nonetheless, some effect of the weighting
method is perceptible for much smaller sample sizes. For example,
Fig. \ref{fig:nullABCB} displays the effect of using the null hypothesis
weights instead of the protein weights on $\tilde{I}_{i}\left(\left(0,\infty\right),\left\{ 0\right\} \right)$
for two randomly selected patients from each breast cancer group and
from the healthy group. Even in this extreme case, only one protein
out of 20 in the right-side panel has a different evidence grade (Table
\ref{tab:Grades-of-hypothesis-evidence}) depending on how the weights
are computed.

\begin{figure}
\includegraphics[scale=0.8]{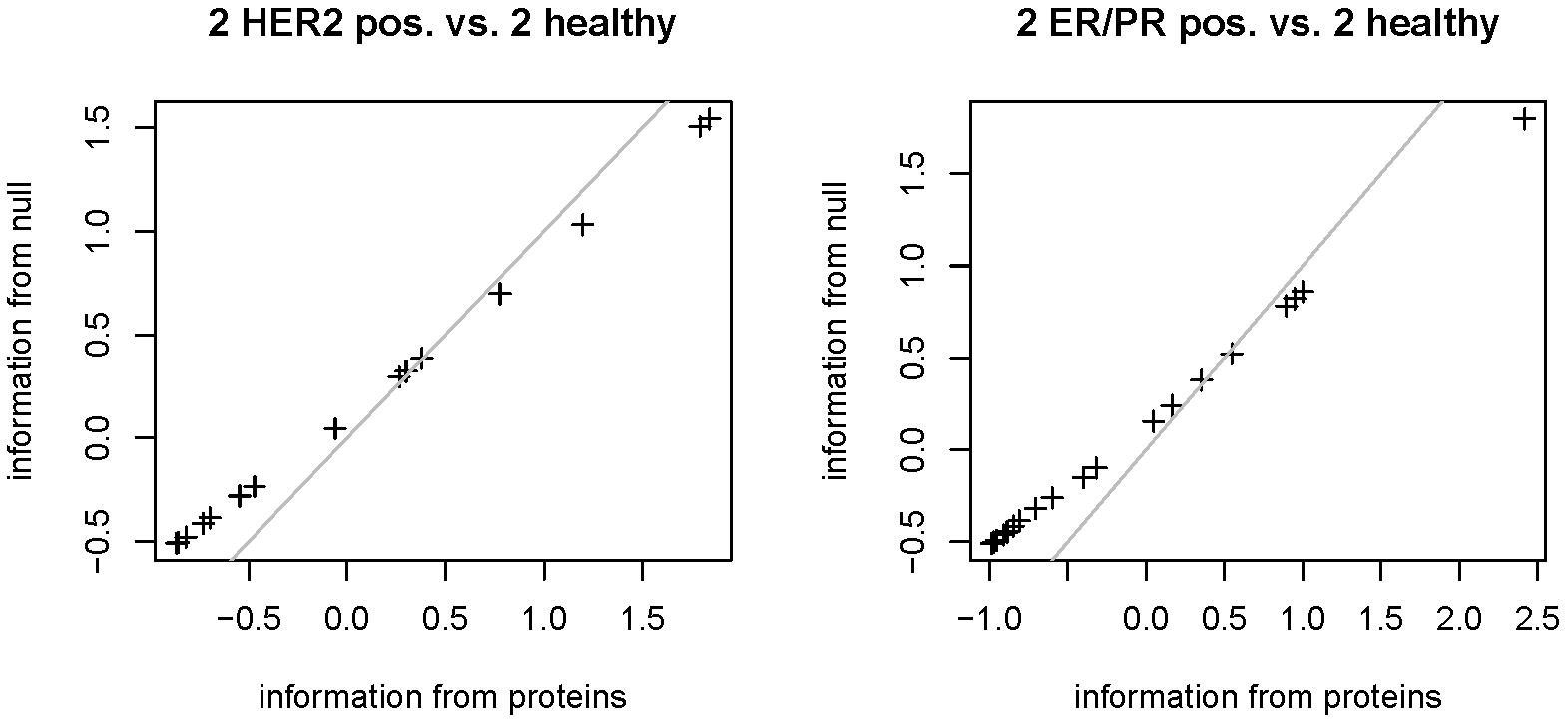}

\caption{\label{fig:nullABCB}Discrimination information (bits) favoring the
hypothesis that the abundance level of a protein differs by disease
status ($n=2$ women per group) using weights from the null hypothesis
versus that using weights from the incidental proteins. }

\end{figure}

\section{\label{sec:Discussion}Discussion}

In both of the case studies of Section \ref{sec:Case-study}, the
use of data associated with comparisons other than the comparison
currently in focus in place of an artificial data point determined
by the null hypothesis has little effect on the information for discrimination.
In the second application, little information was lost for inference
about a single protein were the other 19 absent except when the sample
size was reduced to $n_{i}=4$. Thus, the use of single-observation
weights robustly addresses the infinite-complexity issue with NML
raised in Section \ref{sub:Optimality-observed}. 

The insensitivity to the use of incidental information also suggests
that the NMWL solution to the incidental-information issue raised
in the same section is a measure of evidence that has the same interpretation
for any number of comparisons. By contrast, p-values adjusted to control
error rates and, to a lesser extent, posterior probabilities from
hierarchical Bayesian models, tend to vary so greatly between a single
comparison and a large number of comparisons that they require researchers
to separately build the intuition needed to interpret statistical
reports for small numbers of comparisons, medium numbers of comparisons,
large numbers of comparisons, etc. This shortcoming of traditional
approaches to the multiple comparisons problem is especially glaring
when an article reports various degrees of adjusting p-values for
data types involving very different numbers of features.

As seen in Section \ref{sub:Populations}, the optimal information
for discrimination can indicate strong evidence for a simple null
hypothesis. While in principle the Bayes factor can also favor the
null hypothesis, prior distributions commonly used in practice often
can provide only weak Bayes-factor support for a simple null hypothesis
that corresponds to the data-generating distribution \citep{Johnson2010143}.
The ability of the information for discrimination to indicate whether
the evidence in the data is strongly in favor of the alternative hypothesis,
strongly in favor of the null hypothesis, or insufficient to strongly
favor either hypothesis (Table \ref{tab:Grades-of-hypothesis-evidence})
guards against the prevalent misinterpretation of a high p-value as
evidence for a null hypothesis. More important, the discrimination
information provides scientists a reliable tool designed to objectively
answer the questions they ask of their data.

\section*{Acknowledgments}

The author thanks Corey Yanofsky for comments on the manuscript. \texttt{Biobase}
\citep{RefWorks:161} facilitated data management. This research was
partially supported  by the Canada Foundation for Innovation, by
the Ministry of Research and Innovation of Ontario, and by the Faculty
of Medicine of the University of Ottawa.

\begin{flushleft}
\bibliographystyle{elsarticle-harv}
\bibliography{refman}

\par\end{flushleft}
\end{document}